\documentclass[11pt,a4paper,reqno]{amsart}
\usepackage[T1]{fontenc}
\usepackage{lmodern}
\usepackage[margin=29mm]{geometry}
\usepackage{amsmath,amssymb,amsthm,mathtools}
\usepackage{microtype}
\usepackage{booktabs}
\usepackage{graphicx}
\usepackage[hidelinks]{hyperref}

\newtheorem{theorem}{Theorem}[section]
\newtheorem{proposition}[theorem]{Proposition}
\newtheorem{lemma}[theorem]{Lemma}

\newtheorem{remark}[theorem]{Remark}
\newtheorem{example}[theorem]{Example}
\newcommand{\Z}{\mathbb{Z}}
\newcommand{\U}{\mathcal{U}}

\newcommand{\Mord}{\mathcal{A}}
\newcommand{\Hit}{\mathsf{G}}

\title[Erd\H{o}s--Straus: inversion of Mordell classes]{Erd\H{o}s--Straus Conjecture: Modular Inversion in Mordell Classes}
\author{Andr\'es Ventas}
\date{29 October 2026}
\subjclass[2020]{11D68, 11N13, 11Y16}
\keywords{Erd\H{o}s--Straus conjecture, Mordell classes, ceiling continued fractions, modular inversion, prime number theorem in arithmetic progressions}

\begin{document}
\begin{abstract}
We examine modular inversion of strict Grid certificates in a ceiling
continued-fraction construction for the Erd\H{o}s--Straus equation.
The six Mordell classes modulo $840$ form four inversion orbits.
Prime-input pattern frequencies in the pairs
$121\leftrightarrow361$ and $289\leftrightarrow529$
have equal asymptotic limits for every fixed cutoff.
\end{abstract}
\maketitle

\section{Introduction}

The Erd\H{o}s--Straus conjecture asks whether every integer $n\ge2$
admits positive integers $x,y,z$ satisfying
\begin{equation*}
 \frac{4}{n}=\frac{1}{x}+\frac{1}{y}+\frac{1}{z}.
\end{equation*}
A solution for each prime input suffices for the general conjecture.
The six classical Mordell classes \cite{Mordell} are
\begin{equation*}
 \Mord=\{1,121,169,289,361,529\}\pmod{840}.
\end{equation*}

The ceiling continued-fraction approach to this equation is described by Ventas
in~\cite{VentasFCT}.  We study here the strict Grid subfamily, for
which inversion has an explicit divisor-complement form.
The main statements concern finite certificates, their modular
symmetry, and the limiting behavior of their frequencies on primes.

The six Mordell classes modulo $840$ form four inversion orbits.
Inversion preserves every Grid-hit indicator below a fixed cutoff,
including the first-hit index and the full finite hit pattern.
We prove that the inverse representative of a Mordell prime, taken at
its maximum Grid cutoff, cannot be a perfect square.
For each fixed cutoff, the prime number theorem in arithmetic progressions
transfers the modular symmetry to equality of limiting prime-pattern
frequencies in inverse classes.
Finite CRT counts and prime data illustrate the result.
The identities do not establish universal Grid coverage or a reduction
of that question to four ordinary prime classes.

\section{Strict Grid certificates}

For $n\equiv1\pmod4$ and $t\ge1$, write $m=4t$ and let
$r$ be the least nonnegative residue of $-n\pmod m$.
Define the strict Grid indicator by
\begin{equation}\label{eq:grid}
 \Hit_t(n)=1
 \quad\Longleftrightarrow\quad
 1<r<m,\quad r\mid m+1.
\end{equation}
The definition uses $n/m$ \emph{without reduction}.
Since $n\equiv1\pmod4$, every admissible remainder satisfies
$r\equiv3\pmod4$.

Set $i=(n+r)/m$ and $d=(m+1)/r$.  The corresponding ceiling
continued fraction and Erd\H{o}s--Straus certificate are
\begin{equation}\label{eq:certificate}
 \begin{aligned}
  \frac{n}{m}
   &=\left\lceil i,d,r\right\rceil
    =i-\frac{1}{d-\frac{1}{r}},\\
  k&=\frac{n+i}{r}\in\Z_{>0},\\
  \frac{4}{n}
   &=\frac{1}{ti}+\frac{1}{tik}+\frac{1}{ntk}.
 \end{aligned}
\end{equation}
The divisibility $r\mid n+i$ follows from $dr=m+1$ and
$n=mi-r$.  The unit-fraction identity is verified by
\begin{equation*}
 \begin{aligned}
  r&=4ti-n,\\
  n(k+1)+i
    &=\frac{n(n+i)}{r}+n+i
     =\frac{(n+r)(n+i)}{r}=4tik,\\
  \frac{1}{ti}+\frac{1}{tik}+\frac{1}{ntk}
    &=\frac{n(k+1)+i}{ntik}=\frac{4}{n}.
 \end{aligned}
\end{equation*}
The brackets $\lceil\cdots\rceil$ denote a ceiling continued fraction,
not the ceiling of a real number.

\begin{lemma}[Maximum index of a strict Grid hit]\label{lem:cutoff}
For every integer $n>1$ with $n\equiv1\pmod4$,
\begin{equation}\label{eq:maximum}
 \Hit_t(n)=1\quad\Longrightarrow\quad t\le\frac{n-1}{2}.
\end{equation}
\end{lemma}
\begin{proof}
Suppose $t>(n-1)/2$.  Since $n$ is odd and $t$ is an
integer, $t\ge(n+1)/2$.  In this range $m=4t>n$, so
$r=m-n$ is the least positive remainder of $-n\pmod m$.
If a strict hit existed, then
\begin{equation*}
 \begin{aligned}
  m&\ge2n+2,\qquad r=m-n\ge n+2,\\
  r\mid m+1&\quad\Longrightarrow\quad
  r\mid(m+1)-r=n+1.
 \end{aligned}
\end{equation*}
The last divisibility is impossible because $r>n+1$.
The bound allows $n<m<2n$, where valid certificates have $i=1$.
\end{proof}

\begin{lemma}[Coprimality condition]\label{lem:coprime}
If $\Hit_t(n)=1$, then $\gcd(n,4t)=1$.
\end{lemma}
\begin{proof}
Put $m=4t$ and $g=\gcd(n,m)$.  At a strict hit,
\begin{equation*}
 \begin{aligned}
  r&\equiv-n\pmod m, & g&\mid r,\\
  r&\mid m+1,           & g&\mid m+1,\\
  g&\mid m,             & g&\mid(m+1)-m=1.
 \end{aligned}
\end{equation*}
Thus $g=1$.
\end{proof}

\section{Inversion and Mordell classes}

For an integer $T\ge2$, put
\begin{equation}\label{eq:modulus}
 M_T=\operatorname{lcm}(840,4,8,\ldots,4T).
\end{equation}
For a unit $u\pmod{M_T}$, let $\iota_T(u)$ be the unique
integer in $[1,M_T)$ satisfying $u\iota_T(u)\equiv1\pmod{M_T}$.

\begin{theorem}[Exact inversion of finite Grid patterns]\label{thm:inversion}
Let $u\equiv1\pmod4$ be a unit modulo $M_T$ and
$v=\iota_T(u)$.
Then
\begin{equation}\label{eq:hit-invariance}
 \Hit_t(u)=\Hit_t(v)\qquad(1\le t\le T).
\end{equation}
At every common hit, the remainders satisfy
\begin{equation*}
 r_u r_v=4t+1.
\end{equation*}
Hence the full finite pattern, its hit count, and its first-hit
index when at most $T$ are preserved.
\end{theorem}
\begin{proof}
Write $m=4t$, and suppose $u\equiv-r\pmod m$ with
$1<r<m$ and $r\mid m+1$.  Set $s=(m+1)/r$.
Since $m+1\equiv1\pmod4$ and $r\equiv3\pmod4$,
\begin{equation*}
 \begin{aligned}
  1<s<m,\qquad s&\equiv3\pmod4,\\
  rs&=m+1\equiv1\pmod m,\\
  v&\equiv u^{-1}\equiv(-r)^{-1}\equiv-s\pmod m.
 \end{aligned}
\end{equation*}
Thus $\Hit_t(v)=1$.  Exchange $u$ and $v$ for the converse.
\end{proof}

Inversion modulo $840$ acts on the six classes by
\begin{equation}\label{eq:orbits}
 \begin{aligned}
  1&\longleftrightarrow1,&
  169&\longleftrightarrow169,\\
  121&\longleftrightarrow361,&
  289&\longleftrightarrow529.
 \end{aligned}
\end{equation}
For every fixed $T$, inversion is an involution on the units
modulo $M_T$.  It exchanges the residue classes $121$ and $361$,
and $289$ and $529$, and maps each of the classes $1$ and $169$
to itself.  Being self-inverse modulo $840$ does not mean that every
representative is its own inverse modulo $M_T$.  The bijection
concerns invertible residues, not the sets of prime integers.

\begin{remark}[A composite inverse at the maximum prime cutoff]
\label{rem:1201}
The smallest prime in class $361\pmod{840}$ is $p=1201$.
With $T_p=600$, let
\begin{equation*}
 \begin{aligned}
  M_{600}&=\operatorname{lcm}(840,4,8,\ldots,2400),\\
  c&=\iota_{600}(1201)=\frac{1015M_{600}+1}{1201},\\
  1201c&\equiv1\pmod{M_{600}},\qquad
  c\equiv121\pmod{840}.
 \end{aligned}
\end{equation*}
The integer $c$ has $259$ decimal digits and is composite:
a modular exponentiation gives $2^{c-1}\not\equiv1\pmod c$.
Every prime factor of $c$ exceeds $600$ because
$\gcd(c,M_{600})=1$.
The first two strict hits and their complementary remainders are
\begin{equation*}
 \begin{array}{c|cc|cc}
   4t & r_p & r_c & \dfrac{4t+1}{r_p} & \dfrac{4t+1}{r_c} \\
   \hline
   68 & 23 & 3 & 3 & 23\\
   212 & 71 & 3 & 3 & 71
 \end{array}
\end{equation*}
The corresponding ceiling continued fractions, written for
$n/(4t)$ rather than its reciprocal, are
\begin{equation*}
 \begin{aligned}
  \frac{1201}{68}&=\left\lceil18,3,23\right\rceil,&
  \frac{c}{68}&=\left\lceil\frac{c+3}{68},23,3\right\rceil,\\
  \frac{1201}{212}&=\left\lceil6,3,71\right\rceil,&
  \frac{c}{212}&=\left\lceil\frac{c+3}{212},71,3\right\rceil.
 \end{aligned}
\end{equation*}
Thus a prime in class $361$ and a composite in class $121$
share both hit indices $t=17,53$, with complementary remainders.
\end{remark}

\begin{example}[Three hand-checkable inverse pairs at $T=11$]
\label{ex:small}
The modulus is $M_{11}=110880$.  The following pairs are inverse
modulo $M_{11}$; the first two lie within self-inverse classes:
\begin{equation*}
\begin{array}{c|rr|r|rr}
 a & u & \iota_{11}(u) & (u\iota_{11}(u)-1)/M_{11}
   & r_u & r_{\iota_{11}(u)}\\ \hline
 1   & 5881 & 54601  & 2896 & 15 & 3\\
 169 & 1009 & 81649  & 743  & 3  & 15\\
 361 & 4561 & 103441 & 4255 & 15 & 3
\end{array}
\end{equation*}
All three have their first strict hit at $4t=44$.
The third pair gives, without any reduction,
\begin{equation*}
\begin{aligned}
4561\cdot103441&=4255\cdot110880+1,\\
\frac{4561}{44}&=\left\lceil104,3,15\right\rceil,\\
\frac{103441}{44}&=\left\lceil2351,15,3\right\rceil,\\
15\cdot3&=45=44+1.
\end{aligned}
\end{equation*}
Here $4561$ is prime, whereas $103441=13\cdot73\cdot109$ is
composite.  The first two rows illustrate that $1$ and $169$ are
self-inverse \emph{classes}, although their displayed representatives
are distinct.
\end{example}

\section{Maximum-cutoff inverses and square exclusion}

For a prime $p$ in a Mordell class, define
\begin{equation}\label{eq:maximum-inverse}
 \begin{aligned}
  T_p&=\frac{p-1}{2},\\
  b_p&=\iota_{T_p}(p),\qquad
  pb_p\equiv1\pmod{M_{T_p}}.
 \end{aligned}
\end{equation}
The inverse exists because $p>T_p$ and $p\nmid840$.
Theorem~\ref{thm:inversion} and Lemma~\ref{lem:cutoff} give
\begin{equation}\label{eq:full-cutoff}
 \bigl(\exists t\ge1:\Hit_t(p)=1\bigr)
 \quad\Longleftrightarrow\quad
 \bigl(\exists t\le T_p:\Hit_t(b_p)=1\bigr).
\end{equation}
The cutoff is complete for $p$, not necessarily for the much larger $b_p$.

\begin{proposition}[No square inverse at the maximum prime cutoff]
\label{prop:no-squares}
Let $p\equiv1\pmod8$ be prime and
$T_p\le T<p$.  If $pb\equiv1\pmod{M_T}$, then
$b$ is not a perfect square.  In particular, no $b_p$
for a Mordell prime is a square.
\end{proposition}
\begin{proof}
Since $(-1/p)=1$, some quadratic nonresidue modulo $p$
has a representative $u$ with $1\le u\le T_p$.
In the prime factorization of $u$, at least one prime
$\ell\le T_p$ satisfies $(\ell/p)=-1$.
As $(2/p)=1$, the prime $\ell$ is odd.
Quadratic reciprocity and $pb\equiv1\pmod\ell$ give
\begin{equation*}
 \begin{aligned}
  \left(\frac{p}{\ell}\right)
    &=\left(\frac{\ell}{p}\right)=-1,\\
  \left(\frac{b}{\ell}\right)
    &=\left(\frac{p^{-1}}{\ell}\right)=-1,
 \end{aligned}
\end{equation*}
because $\ell\mid M_T$.  A square cannot be a quadratic
nonresidue modulo $\ell$.
\end{proof}

The cutoff is essential.  For example,
\begin{equation*}
 \begin{aligned}
  M_4&=1680,\\
  1201\cdot961&\equiv1\pmod{1680},\qquad961=31^2.
 \end{aligned}
\end{equation*}
This inversion is well below $T_{1201}=600$.

Every prime factor of $b_p$ exceeds $T_p$.  Nevertheless,
these restrictions do not imply a Grid hit in the shared range.
The existence of a certificate is the unresolved step in a
universal class-reduction argument.

\section{Prime intervals and their inverse representatives}

For $a\in\Mord$, let $a^*$ denote its inverse modulo $840$.
For $T\ge2$, define the prime-input set and its inverse representatives by
\begin{equation}\label{eq:family}
 \begin{aligned}
  \mathcal P_a(T)
    &=\{p\text{ prime}:T<p\le2T,\ p\equiv a\pmod{840}\},\\
  \mathcal B_{a^*}(T)
    &=\{\iota_T(p):p\in\mathcal P_a(T)\}.
 \end{aligned}
\end{equation}
Every element of $\mathcal B_{a^*}(T)$ is in class $a^*$
modulo $840$ and may be prime or composite.

\begin{proposition}[Exact prime-interval correspondence]\label{prop:conditional}
Modular inversion induces a bijection
\begin{equation*}
 \mathcal P_a(T)\overset{\iota_T}{\longleftrightarrow}
 \mathcal B_{a^*}(T).
\end{equation*}
Every strict Grid hit of a prime $p\in\mathcal P_a(T)$ at $4t$
is also a hit of its partner $b=\iota_T(p)\in\mathcal B_{a^*}(T)$
at the same $4t$.  Conversely, every hit of $b$ with $t\le T$
is a hit of $p$ at the same index.
\end{proposition}
\begin{proof}
Every prime $p\in\mathcal P_a(T)$ is a unit modulo $M_T$:
$p>T$ and $p\nmid840$.  Since inversion is involutive,
it is injective and maps $\mathcal P_a(T)$ onto
$\mathcal B_{a^*}(T)$ by definition.
All possible Grid hits of $p\le2T$ satisfy
$t\le(p-1)/2<T$; the certificate correspondence follows
from Theorem~\ref{thm:inversion}.
\end{proof}

Taking the dyadic cuts $T=2^k$, $k\ge1$, covers every
prime $p>2$ in one of the intervals $(T,2T]$.
Therefore, coverage of all primes in class $a$ is equivalent
to a hit at some $t\le T$ for every $b\in\mathcal B_{a^*}(T)$
and every such dyadic $T$.
This is an equivalence of conditions, not an existence theorem.
The criterion requires control of specified \emph{integer}
representatives, which need not be prime.
It does not reduce universal coverage to four classes of prime inputs.

\section{Asymptotic symmetry of prime Grid patterns}

For a fixed $T$, set
\begin{equation*}
 \begin{aligned}
  \boldsymbol\Hit_T(n)&=(\Hit_1(n),\ldots,\Hit_T(n)),\\
  \U_a(T)&=\{u\in(\Z/M_T\Z)^\times:u\equiv a\pmod{840}\},\\
  \U_{a,\varepsilon}(T)&=\{u\in\U_a(T):
                 \boldsymbol\Hit_T(u)=\varepsilon\},
 \end{aligned}
\end{equation*}
where $\varepsilon\in\{0,1\}^T$.
Let $P_{a,\varepsilon}(X,T)$ count primes
$X\le p<2X$ in class $a$ with pattern $\varepsilon$,
and let $P_a(X)$ count all primes in that class.
Set
\begin{equation*}
 \begin{aligned}
  R_a(X,T)&=\frac{P_{a,\boldsymbol 0}(X,T)}{P_a(X)},\\
  \nu_a(T)&=\frac{|\U_{a,\boldsymbol 0}(T)|}{|\U_a(T)|},
  \qquad \boldsymbol 0=(0,\ldots,0).
 \end{aligned}
\end{equation*}
These are, respectively, the observed prime non-coverage proportion
and the exact proportion among invertible residues at cutoff $T$.

\begin{theorem}[Equality of limiting prime-pattern laws]
\label{thm:asymptotic}
For $a\in\Mord$, $a^*\equiv a^{-1}\pmod{840}$,
and every fixed $T\ge2$ and $\varepsilon\in\{0,1\}^T$,
\begin{equation}\label{eq:limiting-law}
 \begin{aligned}
  \lim_{X\to\infty}\frac{P_{a,\varepsilon}(X,T)}{P_a(X)}
  &=\frac{|\U_{a,\varepsilon}(T)|}{|\U_a(T)|}\\
  &=\frac{|\U_{a^*,\varepsilon}(T)|}{|\U_{a^*}(T)|}\\
  &=\lim_{X\to\infty}
     \frac{P_{a^*,\varepsilon}(X,T)}{P_{a^*}(X)}.
 \end{aligned}
\end{equation}
\end{theorem}
\begin{proof}
Inversion preserves each finite hit pattern and bijects the
corresponding invertible residue classes.  Set
\begin{equation*}
\begin{aligned}
|\U_{a,\varepsilon}(T)|&=|\U_{a^*,\varepsilon}(T)|=:h,\\
|\U_a(T)|&=|\U_{a^*}(T)|
   =\frac{\varphi(M_T)}{\varphi(840)}=:N.
\end{aligned}
\end{equation*}
For fixed $T$, the prime number theorem in arithmetic
progressions~\cite{MontgomeryVaughan} distributes primes uniformly
among the invertible residues modulo $M_T$.  Hence
\begin{equation*}
\begin{aligned}
\lim_{X\to\infty}\frac{P_{a,\varepsilon}(X,T)}{P_a(X)}
 &=\frac{h}{N}
 =\lim_{X\to\infty}\frac{P_{a^*,\varepsilon}(X,T)}{P_{a^*}(X)}.
\end{aligned}
\end{equation*}
\end{proof}

In particular,
\begin{equation}\label{eq:noncoverage}
 \lim_{X\to\infty}R_a(X,T)
 =\nu_a(T)=\nu_{a^*}(T)
 =\lim_{X\to\infty}R_{a^*}(X,T).
\end{equation}
The equality is exact for each fixed $T$; it does not assert
uniformity for cutoffs growing with $X$.

\begin{remark}[Prime statistics without prime-to-prime inversion]
\label{rem:prime-stats}
Inversion at a fixed cutoff acts on all \emph{invertible residues},
not specifically on primes.  An individual prime in class $121$
may invert to a composite in class $361$ (and conversely), as in
Remark~\ref{rem:1201}.  Nevertheless, primes become equidistributed
among the invertible residue classes of each fixed modulus.
Theorem~\ref{thm:asymptotic} compares the \emph{prime frequencies}
on the two sides through these equal modular populations;
it does not pair the individual primes. 
\end{remark}

\section{Computational measurements}

The counts use the unreduced residue test in~\eqref{eq:grid}.
For the Mordell classes, the only possible hits through $T=17$
are at $t=11$ and $t=17$.  At these indices there are two
admissible nonzero residues modulo $11$ and $17$, respectively.
CRT gives independent hit proportions $2/10$ and $2/16$, hence
\begin{equation*}
\begin{aligned}
 \nu_a(11)&=1-\frac{2}{10}=\frac{4}{5},\\
 \nu_a(17)&=\left(1-\frac{2}{10}\right)
                 \left(1-\frac{2}{16}\right)
              =\frac{7}{10},\\
 46\,080\cdot\frac{7}{10}&=32\,256.
\end{aligned}
\end{equation*}
These are exact proportions of invertible residues without any hit
\emph{through the indicated cutoff}, not claims of global non-coverage.
The numerical agreement between distinct inverse pairs at these
cutoffs is verified directly, not implied by the inversion theorem.
\begin{table}[htbp]
\centering\small
\caption{Exact CRT counts of invertible residues at fixed cutoffs.
All counts in the last three columns are \emph{per class} in the indicated pair.}
\label{tab:crt}
\begin{tabular}{@{}r l r r r r r@{}}
\toprule
$T$ & Inverse pair & $M_T$ & Units & No hit through $T$ & Hit through $T$
 & $\nu_a(T)$\\
\midrule
11 & $121\leftrightarrow361$ & $110\,880$ & $120$ & $96$ & $24$ & $0.80$\\
11 & $289\leftrightarrow529$ & $110\,880$ & $120$ & $96$ & $24$ & $0.80$\\
17 & $121\leftrightarrow361$ & $49\,008\,960$ & $46\,080$ & $32\,256$ & $13\,824$ & $0.70$\\
17 & $289\leftrightarrow529$ & $49\,008\,960$ & $46\,080$ & $32\,256$ & $13\,824$ & $0.70$\\
\bottomrule
\end{tabular}
\end{table}

For $10^7\le p<10^8$, the enumerated populations were
$26\,469$ primes in class $121$, $26\,437$ in $361$,
$26\,453$ in $289$, and $26\,463$ in $529$.
Table~\ref{tab:primes} reports observed non-coverage percentages.
\begin{table}[htbp]
\centering\small
\caption{Proportion (\%) of primes with no hit through $T$;
$10^7\le p<10^8$.}
\label{tab:primes}
\begin{tabular}{@{}r r r r r@{}}
\toprule
$T$ & $121$ & $361$ & $289$ & $529$\\
\midrule
11 & 79.91 & 79.97 & 79.72 & 79.84\\
17 & 70.04 & 69.92 & 69.70 & 69.91\\
26 & 46.86 & 47.10 & 46.88 & 47.09\\
41 & 38.97 & 39.29 & 39.26 & 39.21\\
100 & 23.00 & 23.45 & 19.98 & 20.08\\
500 & 4.76 & 4.86 & 3.28 & 3.33\\
\bottomrule
\end{tabular}
\end{table}

\begin{figure}[htbp]
 \centering
\includegraphics[width=.95\linewidth]{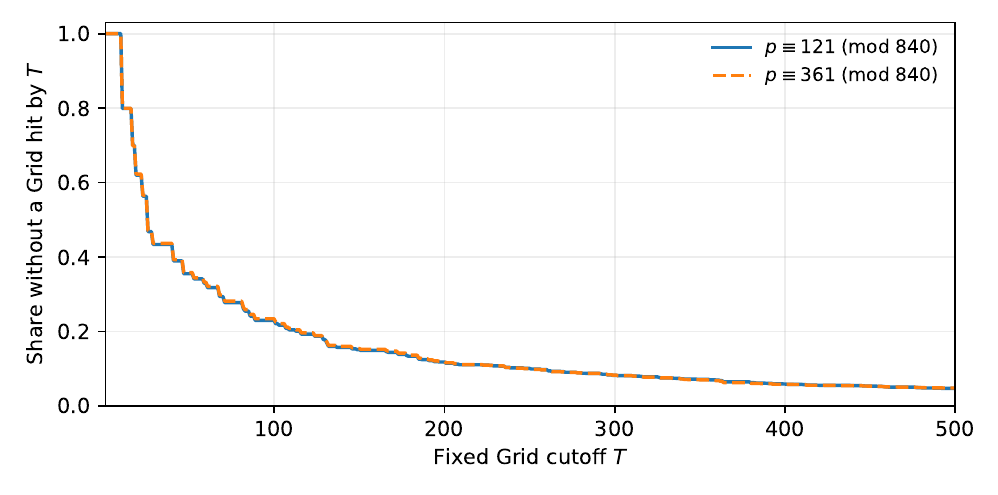}
 \caption{Observed non-coverage for primes in inverse classes
 $121$ and $361$, for $10^7\le p<10^8$.}
 \label{fig:curves}
\end{figure}

The counts were obtained by enumerating primes in $[10^7,10^8)$,
filtering by residue modulo $840$, and evaluating the unreduced
condition~\eqref{eq:grid} for each $1\le t\le T$.
No program or data file is part of the present manuscript.
The differences at finite $X$ do not give uniform error bounds
for cutoffs growing with $X$.

\section{Limits of multiplicative inheritance}

For a composite $n=uv$, an Erd\H{o}s--Straus representation
for $u$ lifts to one for $n$ by multiplying the denominators
by $v$.  This does not, in general, produce a \emph{strict}
Grid hit of $n$: the inherited FCT may require reduction.
Lemma~\ref{lem:coprime} makes the obstruction explicit.

For example,
\begin{equation*}
 \begin{aligned}
  2641&=19\cdot139\equiv121\pmod{840},\\
  \frac{2641}{3336}&=\frac{19}{24}
      =\left\lceil1,5,5\right\rceil,\\
  \gcd(2641,3336)&=139,\\
  (-2641)\bmod3336&=695,\qquad695\nmid3337.
 \end{aligned}
\end{equation*}
An exact search through $t\le1320$ finds no strict Grid
hit for $2641$, although the integer has an ordinary
Erd\H{o}s--Straus representation.
The equality of rational continued fractions after reduction
does not transfer the unreduced Grid-divisibility condition.
For instance, $2641$ is invertible modulo $M_{18}=49\,008\,960$,
with inverse $6365041$ in class $361$.
Both have no strict hit for $t\le18$, as required by the finite
inversion theorem.  The latter has a hit at $t=339$, outside
this shared cutoff; $2641$ is no longer invertible modulo $M_T$
when $T\ge19$.  Thus the bijection does not match global
coverage of fixed integers.

The inversion theorem preserves \emph{existing} strict hits,
not arbitrary representations inherited from factors.
Even a hypothetical proof of Grid coverage for all
non-square integers in class $121$ would not by itself
show that each certificate occurs within the transfer cutoff
attached to an inverse prime in class $361$.
Thus the four inversion orbits give four statistical classes
at each fixed cutoff, not an unconditional reduction of six
prime-coverage problems to four.

\section{Conclusion}
Strict Grid inversion preserves finite hit patterns for each
unit modulo $M_T$.  The Mordell classes form the four orbits
$\{1\}$, $\{169\}$, $\{121,361\}$, and $\{289,529\}$.
For each fixed cutoff, the two nontrivial pairs have identical
limiting prime-pattern frequencies.  This statistical equivalence
does not imply universal Grid coverage in either class.

\section*{AI assistance and attribution}

This manuscript and the underlying research were developed
with extensive assistance from ChatGPT (OpenAI, GPT-6).
The AI participated substantively in the discussion and
formalization of the modular inversion argument, proof steps, Python code and data checks, literature-search
suggestions, mathematical notation, and the structure and
English wording of this manuscript.
Claude (Anthropic) was consulted, separately from
ChatGPT, for an independent check of the proofs and examples,
The mathematical direction, research questions, choices of
examples, and interpretation of the research program were
actively directed by the human author.
The AI systems are not listed as authors.
The human author retains responsibility for checking all
mathematical claims, computations, and references in the
submitted work.

\end{document}